\documentclass[12pt]{article}
\usepackage{amsmath}
\usepackage{float}
\usepackage{amsmath,amsfonts,amssymb,mathtools}
\usepackage{amsthm}
\usepackage{enumitem}
\usepackage{amsthm}
\usepackage{amssymb}
\usepackage[utf8]{inputenc}
\usepackage{graphics}
\usepackage{graphicx}
\usepackage{enumerate} 
\usepackage{multicol}
\usepackage{lmodern}
\usepackage{marvosym} 
\usepackage{lipsum}
\usepackage{mwe}
\usepackage{caption}
\usepackage{subcaption} 
\newtheoremstyle{case}{}{}{}{}{}{:}{ }{}
\theoremstyle{case}

\newcommand{\be}{\begin{equation}}
\newcommand{\ee}{\end{equation}}
\newcommand{\ben}{\begin{eqnarray*}}
\newcommand{\een}{\end{eqnarray*}}
\newtheorem{examp}{\sc Example}
\newtheorem{remk}{\sc Remark}
\newtheorem{corol}{\sc Corollary}
\newtheorem{lemma}{\sc Lemma}
\newtheorem{theorem}{\sc Theorem}
\newtheorem{defn}{\sc Definition}
\newcommand{\bt}{\begin{theorem}}
\newcommand{\et}{\end{theorem}}
\newcommand{\bl}{\begin{lemma}}
\newcommand{\el}{\end{lemma}}
\newcommand{\bed}{\begin{defn}}
\newcommand{\eed}{\end{defn}}
\newcommand{\brem}{\begin{remk}}
\newcommand{\erem}{\end{remk}}
\newcommand{\bex}{\begin{examp}}
\newcommand{\eex}{\end{examp}}
\newcommand{\bcl}{\begin{corol}}
\newcommand{\ecl}{\end{corol}}

\newcommand{\NI}{\noindent}

\newcommand{\vsp}{\vskip 0.5em}

\newtheorem{thm}{Theorem}[subsection]

\newtheorem{lem}[thm]{Lemma}

\theoremstyle{definition}
\theoremstyle{remark}

\numberwithin{equation}{section}
\numberwithin{theorem}{section}
\numberwithin{lemma}{section}

\begin{document}
 \title{{Error Bound for the Linear Complementarity Problem using Plus Function}}
\author{Bharat Kumar$^{a, 1}$, Deepmala$^{b, 2}$, A. Dutta$^{c, 3}$ and A. K. Das$^{d, 4}$\\
\emph{\small $^{a,b}$Department of Mathematics, PDPM-Indian Institute of Information Technology}\\
\emph{\small Design and Manufacturing, Jabalpur - 482005 (MP), India}\\
\emph{\small $^{c}$Department of Mathematics, Jadavpur University, Kolkata, 700 032, India}\\
\emph{\small $^{d}$SQC \& OR Unit, Indian Statistical Institute, Kolkata, 700 108, India}\\
\emph{\small $^{1}$Email: bharatnishad.kanpu@gmail.com}\\
\emph{\small $^{2}$Email: dmrai23@gmail.com}\\
\emph{\small $^{3}$Email: aritradutta001@gmail.com}\\
\emph{\small $^{4}$Email: akdas@isical.ac.in} \\
 }
\date{}
\maketitle

\date{}
\maketitle
\begin{abstract}
	\noindent In this article we establish error bound for linear complementarity problem with $P$-matrix using plus function.  We introduce a fundamental quantity associated with a $P$-matrix and show how
this quantity is useful in deriving error bounds for the linear complementarity
problem of the $P$-type. We also obtain (upper and lower) bounds for the quantity introduced. 
\end{abstract}
\NI{\bf Keywords:} Linear complementarity problem, plus function, error bound, relative error bound. \\
\NI{\bf AMS subject classifications:} 90C33, 15A39, 15A24, 15A60, 65G50. 

\footnotetext[1] {Corresponding author}

\section{Introduction}
The error bounds are important to consider a measure
by which the approximate solution fails to be in the
solution set and to obtain the convergence rates of
different approaches. The error bound decides stopping criteria in terms of convergence analysis for iterative method. It also plays an important role in sensitivity analysis. In linear complementarity problem, use of error bound is focused not only for the bounds of the solution but also for the convergence rate of the iterative method applied to find the solution. Mathias and Pang \cite{mathias1990error} established error bound for linear complementarity problem with $P$-matrix. Here we propose a new error bound for linear complementarity problem with $P$-matrix based on plus function.

The linear complementarity problem finds a real valued vector that satisfies a particular system of inequalities and a complementarity condition. The problem is defined as follows:
 
 Given $A\in R^{n \times n}$ and a vector $q \in R^n $. Now consider the linear complementarity problem as to find a vector $z \in R^n$ such that 
 \begin{equation}\label{1}
     Az+q \geq 0, \ \ \ z \geq 0,
     \end{equation}
     \begin{equation}\label{2}
         z^T(Az+q)=0.
 \end{equation}
  This problem is denoted as LCP$(A,q)$ for which we find a vector $z\in R^n $ satisfying the  inequalities (\ref{1}) as well as complementarity condition (\ref{2})  or show that no such $z$ exists.
  The feasible set of LCP$(A,q)$ is defined as FEA$(A,q)$ $=\{z:Az+q \geq 0, z \geq 0\}$ and the solution set is defined as SOL$(A,q)$ $=\{z \in FEA(A,q):z^T(Az+q)=0\}$.
\\
 
For the recent study on the linear complementarity problem and applications see \cite{das2017finiteness}, \cite{articlee14}, \cite{bookk1}, \cite{articlee7} and references therein. For details of several matrix classes in complementarity theory, see \cite{articlee1}, \cite{articlee2}, \cite{articlee9}, \cite{articlee17}, \cite{article1}, \cite{mohan2001more}, \cite{article12}, \cite{article07}, \cite{dutta2022column} and references cited therein. The problem of computing the value vector and optimal stationary strategies for structured stochastic games for discounted and undiscounded zero-sum games and quadratic multi-objective programming problem are formulated as linear complementary problems. For details see \cite{articlee18}, \cite{mondal2016discounted}, \cite{neogy2005linear} and \cite{neogy2008mixture}. The complementarity problems are considered with respect to principal pivot transforms and pivotal method to its solution point of view. For details see \cite{articlee8}, \cite{articlee10}, \cite{das1} and \cite{neogy2012generalized}.

For the LCP$(A, q)$, one important issue is to study the related error bound, which is an inequality that bounds the distance from vectors to the solution set of the LCP$(A,q)$, in terms of some residual function. Now consider $S=$SOL$(A,q)$. The nonegative function $\bar{r}:S \to R_{+}$ is said to be a residual function for LCP$(A,q)$ if it satisfies the property that $\bar{r}(z)=0$
be a residual function for the LCP$(A, q)$ if it satisfies the property that $\bar{r}(z) = 0$
if and only if $z \in S$.
In recent years, many researchers are concerned with
global error bounds for the linear complementarity problem and related
mathematical programs. For details see \cite{mathias1990error}, \cite{chen2006computation}, \cite{garcia2010comparison} and \cite{dai2013new}.

In this paper we consider linear complementarity problem with $P$-matrix. Our aim is to derive global upper and lower error bounds of solution of LCP$(A,q)$ using plus function. In section 2,  some basic notations and results are presented. In section 3, we derive upper and lower bounds of solution of LCP$(A,q)$.  In the last  section, give some examples to illustrate the bounds of LCP$(A,q)$.

\section{Preliminaries}
 We begin by introducing some basic notations used in this paper. We consider matrices and vectors with real entries. $R^n$  denotes the $n$ dimensional real space, $R^n_+$ and $R^{n}_{++}$ denote the nonnegative and positive orthant of $R^n.$ We consider vectors and matrices with real entries. Any vector $x\in R^{n}$ is a column vector and  $x^{T}$ denotes the row transpose of $x.$ $e$ denotes the vector of all $1.$ $x_i$ denotes the $i$-th component of the vector $x\in R^n.$\\
For a given $z\in R^n$, we define $\|z\|_{\infty}=\max\limits_{i}|z_i|$.\\
For a given matrix $A \in R^{n \times n}$, we define $\|A\|_{\infty}=\max\limits_{i,j}|A_{ij}|$.\\
 Let ${\|z\|^2 _A}=\max\limits_{ i }z_{i}(Az)_{i}$. Now the quantity $\beta(A)$ is defined as
\begin{equation}\label{def1}
    \beta(A)=\min\limits_{\|z\|_{\infty}=1} \|z\|^{2}_A
\end{equation} which is finite and positive.  For an arbitrary vector $z \in R^n$,
the following inequality holds:

\begin{equation}\label{eq2}
 \max\limits_{i} ~z_{i}(Az)_{i}\geq~ \beta(A)\|z\|^{2}_{\infty}.
\end{equation}
For details see \cite{mathias1990error}.
 Now we give some definitions and lemmas which will be used in the next section.
 \begin{defn}
  Consider $z \in R$. Then the plus function is defined as  $z_{+}=\max(0,z)$. 
 \end{defn}
For $z\in R^n$, the plus function $z_{+}$ is also defined as $(z_{+})_{i}=(z_i)_{+} \ \forall \ i,$ where $(z_i)_{+}=\max(0,z_i).$
From the definition of $z_{+}$ it is clear that $z_{+}\geq 0.$

\begin{defn}\cite{pang1995complementarity}
$A \in R^{n \times n}$ is called $P$-matrix if and only if for every  $z\in R^n \backslash \{0\}$
\begin{equation}\label{eq1}
 \max_{ i }z_{i}(Az)_{i} ~ > ~ 0;
 \end{equation}
\end{defn}
\begin{lem}\label{lem1}\cite{mathias1990error}
Let $A\in R^{n \times n}$ be  a $P$-matrix and let $x $ denote the unique
solution of LCP$(A,q)$. Then
\begin{equation}\label{eq7}
{\beta(A^{-1})}\|(- q)_{+}\|_{\infty} \leq \|z\|_{\infty}\leq  {\beta(A)^{-1}}\|(- q)_{+}\|_{\infty}.
\end{equation}
\end{lem}
Now our aim is to bound the error for the linear complementarity problem LCP$(A,q)$, where $A\in R^{n \times n}$ is a $P$-matrix.
\section{Main Results}

\begin{theorem}\label{lem2}
Let $A\in R^{ n \times n}$ be a $P$-matrix. Let $z$ denote the unique solution
of LCP$(A,q)$ and let $d \in R^n$ be an arbitrary vector. Then, 

\begin{equation}\label{eqq8}
    \frac{1}{1+\|A\|_{\infty}}\|(d-(d-(Ad+q))_{+})\|_{\infty}\leq \|z-d\|_{\infty}\leq \frac{1+\|A\|_{\infty}} {\beta(A)}\|(d-(d-(Ad+q))_{+})\|_{\infty}.
\end{equation}
\end{theorem}
\begin{proof}
 Let $p=(d-(d-(Ad+q))_{+})$ and $l=q+Az$. Consider the vector 
 \begin{center}
 $s=d-p=d-(d-(d-(Ad+q))_{+})=(d-(Ad+q))_{+}\geq 0$.\\
 \end{center}
 Let $c=q+(A-I)p+As.$ This implies that
 \begin{equation*}
 c=q+(A-I)p+As=q+Ap-p+Ad-Ap=q-p+Ad\\
   =q-(d-(d-(Ad+q))_{+})+Ad=q-d+(d-(Ad+q))_{+}+Ad.
 \end{equation*}
 Now we show that the vectors $c$ and $s$ satisfy complementarity condition. \\If $d_i\geq (Ad+q)_i,$ \begin{center}
 $((d-(Ad+q))_{i})_{+}=d_i-(Ad+q)_i$.
 \end{center}
 Then 
 \begin{center}
     $c_i=q_i-d_i+d_i-(Ad+q)_i+(Ad)_i=0$ and $s_i=((d-(Ad+q))_{i})_{+}=d_i-(Ad+q)_i\geq 0.$
     \end{center}
     In another way, if $ d_i\leq (Ad+q)_i, $
     \begin{center}$((d-(Ad+q))_{i})_{+}=0.$ 
     \end{center}
     This implies that 
     \begin{center}
     $s_i=0$ and $c_i=q_i-d_i+((d-(Ad+q))_{i})_{+}+(Ad)_{i}=(Ad)_i+q_i-d_i\geq 0.$ 
     \end{center}    
     Considering both the cases we obtain the following pair of inequalities and complementarity condition.
 
 \begin{equation}
     s\geq 0,~~~~ c=q+(A-I)p+As\geq 0, 
     \end{equation}
     \begin{equation}
     s^Tc=0.
 \end{equation}
 Now, for each $i$ we have
 \begin{center}
 $(s-z)_i(c-l)_i=s_i c_i+z_i l_i-z_i c_i-s_i l_i\leq 0$.
 \end{center}
 Therefore, we have\\
 \begin{equation*}
 0\geq (s-z)_i(c-l)_i=(d-p-z)_i(q+Ap-p+As-q-Az)_{i}     = (d-p-z)_i(q+Ap-p+Ad-Ap-q-Az)_{i}= (d-p-z)_i (-p+A(d-z))_{i}\\
          =-(d-z)_i p_i-p_i (A(d-z))_i + {p_i}^2+(d-z)_i (A(d-z))_i\\
     geq -(d-z)_ip_i-p_i(A(d-z))_i+(d-z)_i(A(d-z))_i.
 \end{equation*}
  This implies that,\\
  \begin{center}
 $(d-z)_i(A(d-z))_i\leq  (d-z)_i p_i+p_i(A(d-z))_i$
 \end{center}
 
 \NI In particular for the index $i$, we have\\
 \begin{center}
 $(d-z)_i(A(d-z))_i=\max\limits_j(d-z)_j(A(d-z))_j$.
 \end{center}
 Now from the condition ($\ref{eq2}$), we have\\
 \begin{center}
 $\max\limits_{i} ~z_{i}(Az)_{i}\geq~ \beta(A)\|z\|^{2}_{\infty}.$
 \end{center}
 Hence 
 \begin{center}
    $(d-z)_i(A(d-z))_i\geq \beta(A)\|d-z\|^{2}_{\infty}$.
    \end{center}
 Therefore,
 
 \begin{equation*}
 \beta(A)\|d-z\|^{2}_{\infty}\leq (d-z)_i(A(d-z))_i\leq (d-z)_i p_i + p_i(A(d-z))_i \\
 =((I+A)(d-z))_i \ p_i \leq (1+\|A\|_{\infty})\|p\|_{\infty}\|d-z\|_{\infty}.
 \end{equation*}

 Hence 
 \begin{center}
     $\|d-z\|^{2}_{\infty}\leq\frac{(1+\|A\|_{\infty})}{\beta(A)}\|p\|_{\infty}\|d-z\|_{\infty}.$
     \end{center}
 \vsp
To prove the left-hand inequality of ($\ref{eqq8}$), consider an arbitrary index $i$ for which
 $p_i > 0$ and  $l_i = 0$. Then $(Az)_i=-q_i.$ \\
In this case, if $p_i>0,$ then 
\begin{center}
    $p_i=d_i-((d-(Ad+q))_+)_i = d_i-((d-(Ad+q))_i)_+>0$. 
    \end{center}
Since $l_i=0,$ the inequality $(Ad+q)_i\geq d_i$, implies that 
\begin{center}
$p_i=d_i\leq (Ad+q)_i=(A(d-z))_i$ 
\end{center}
and  the another inequality
 $(Ad+q)_i\leq d_i$, implies that
 \begin{equation*}
 p_i=d_i-d_i+(Ad+q)_i=(Ad+q)_i=(A(d-z))_i. 
 \end{equation*}
 Hence considering the case $p_i > 0$ and  $l_i = 0$, we obtain \\
 \begin{equation*}
 \lvert p_i \rvert=p_i \leq (A(d-z))_i \leq \|A\|_\infty \|(d-z)\|_\infty.
\end{equation*}\\
If $z_i = 0$, then $s=d-p\geq 0$ implies that $p_i \leq d_i.$ \\
Therefore,
\begin{equation*}
\lvert p_i \rvert=p_i \leq d_i-z_i \leq \|(d-z)\|_{\infty},
 \text{when} z_i = 0, p_i>0. 
 \end{equation*}
Thus considering all the cases, we conclude that \\
\begin{center}
$|p_i|\leq (1+\|A\|_\infty)\|(d-z)\|_\infty$.
\end{center}
Now we consider the case $p_i\leq 0.$\\
Let $l_i=0.$ Then 
\begin{center}
$p_i=d_i-((d-(Ad+q))_+)_i = d_i-((d-(Ad+q))_i)_+<0.$
\end{center}
Considering both the cases $(Ad+q)_i\leq d_i$ and $(Ad+q)_i\geq d_i$, we obtain\\
\begin{center}
$|p_i|=-p_i\geq -(A(d-z))_i$, which implies that \\
$|p_i|\leq (A(d-z))_i\leq \|A\|_\infty \|(d-z)\|_\infty.$
\end{center}
If $z_i = 0$, then $|p_i|=-p_i\geq -d_i+z_i $, which implies that $|p_i|\leq \|(d-z)\|_{\infty}.$
Hence considering the cases $p_i<0, l_i=0, z_i=0$, we conclude that \\
\begin{center}
$|p_i|\leq (1+\|A\|_\infty)\|(d-z)\|_\infty$, where $p_i=(d-(d-(Ad+q))_+)_i.$
\end{center} 
This inequality $|p_i|\leq (1+\|A\|_\infty)\|(d-z)\|_\infty$ implies that
\begin{center}
$\|(d-z)\|_\infty\geq$ $\frac{|p_i|}{(1+\|A\|_\infty)}$, for an arbitrary index $i$.
\end{center}
Hence 
\begin{center}
$\|(d-z)\|_\infty\geq$ $\frac{\|p\|}{(1+\|A\|_\infty)}$. 
\end{center}
Hence we conclude that 
\begin{equation*}
\frac{\|d-(d-(Ad+q))_+\|_{\infty}}{(1+\|A\|_\infty)}\leq \|(z-d)\|_\infty \leq 
\frac{(1+\|A\|_{\infty})}{\beta(A)}\|d-(d-(Ad+q))_+\|_{\infty}.
\end{equation*}
\end{proof}
\begin{remk}
The quantity $\|(d-(d-(Ad+q))_{+})\|_{\infty}$ in the expression ($\ref{eqq8}$) is the residue of
the vector $d$. When $d = 0$, this residue is equal to
the quantity $\|( - q)_+\|_{\infty} $. Now we deduce the relative error bound.
 
\end{remk}
\begin{theorem}\label{lem3}
 Let $A \in R^{n \times n}$ be a $P$-matrix. Let $z$ denote the unique
solution of LCP$(A,q)$ and let $d\in R^n $ be an arbitrary vector. Assume that
$(- q)_+ \neq 0$. Then,

\begin{equation*}
    \frac{\beta(A)}{1+\|A\|_{\infty}}\frac{\|(d-(d-(Ad+q))_{+})\|_{\infty}}{\|( - q)_+\|_{\infty}}
\end{equation*}

\begin{equation}\label{eq9}
\leq \frac{\|z-d\|_{\infty}}{\|z\|_{\infty}}\leq \frac{1+\|A\|_{\infty}}{\beta(A^{-1})\beta(A)}\frac{\|(d-(d-(Ad+q))_{+})\|_{\infty}}{\|( - q)_+\|_{\infty}}.
\end{equation}
\end{theorem} 
\begin{proof}
 From theorem \ref{lem2}, we obtain,
 \begin{equation*}\label{eq8}
    \frac{1}{1+\|A\|_{\infty}}\|(d-(d-(Ad+q))_{+})\|_{\infty}\leq \|z-d\|_{\infty}\leq \frac{1+\|A\|_{\infty}} {\beta(A)}\|(d-(d-(Ad+q))_{+})\|_{\infty}.
\end{equation*}
and from lemma \ref{lem1}, it is given that,
\begin{equation*}\label{eq7}
{\beta(A^{-1})}\|(- q)_{+}\|_{\infty} \leq \|z\|_{\infty}\leq  {\beta(A)^{-1}}\|(- q)_{+}\|_{\infty}.
\end{equation*}
Now combining these two inequalities, we obtain

\begin{equation*}
    \frac{\beta(A)}{1+\|A\|_{\infty}}\frac{\|(d-(d-(Ad+q))_{+})\|_{\infty}}{\|( - q)_+\|_{\infty}}
\end{equation*}

\begin{equation*}\label{eq9}
\leq \frac{\|z-d\|_{\infty}}{\|z\|_{\infty}}\leq \frac{1+\|A\|_{\infty}}{\beta(A^{-1})\beta(A)}\frac{\|(d-(d-(Ad+q))_{+})\|_{\infty}}{\|( - q)_+\|_{\infty}}.
\end{equation*}
 
\end{proof}

\begin{theorem}\label{cor1}
Let $A\in R^{n\times n}$ be a $P$-matrix. Then
\begin{equation*}
    \beta(A)\leq \sigma(A),
\end{equation*}
where $\sigma(A)=\min\{\gamma(A_{\mu\mu}):\mu \subseteq \{1,2,\ldots ,n\} \},$ \ $\gamma(A_{\mu \mu})$ denotes the smallest  eigenvalue of the principal submatrix $A_{\mu \mu}$.
\end{theorem}
\begin{proof}
 Let  $\sigma(A)=\min\{\gamma(A_{\mu\mu}):\mu \subseteq \{1,2,\ldots ,n\} \},$ \ $\gamma(A)$ denotes the smallest  eigenvalue of the principal submatrix $A_{\mu \mu}$. By this definition of $\sigma(A)$, the matrix $(A-\sigma(A)I)$ cannot be a $P$-matrix. Then, there exists a vector ${y}$  such that
 \begin{equation*}
     \max_{i}{y}_{i}((A-\sigma(A) I){y})_{i}\leq 0.
 \end{equation*}
 This implies that
 \begin{center}
$ \max\limits_{i} ({y}_{i}(A{y})_{i}-\sigma(A){{y}_i}^2)\leq 0.$\\
 \end{center}
 \begin{center}
 With $\|{y}\|_{\infty}=1,$ \ $\max\limits_{i} ({y}_{i}(A{y})_{i})\leq \sigma(A).$
 \end{center}
 Now introducing minimum in both side of the above inequality, we obtain,
 \begin{center}
     $\min\limits_{\|{y}\|_{\infty}=1} \max\limits_{i} {y}_{i}(A{y})_{i}\leq\sigma(A).$\\
 \end{center}
 This implies that
 \begin{equation*}
     \alpha(A)\leq \|y\|^2_A\leq \sigma(A).
 \end{equation*}
\end{proof}
\begin{corol}\label{cor2}
Let $A\in R^{n \times n}$ be a nondiagonal $P$-matrix. Let $\lambda \in (0,1)$ be arbitrary. Now define the scalars 
\begin{equation*}
    m=\max_{i\neq j}|A_{ij}|, ~~~~~ h=\frac{m^2}{\sigma(A)}.
\end{equation*}
Define the numbers $\{t_{i}: i \in \{1,\ldots,n\}\}$  
\begin{equation*}
    t_1=\min\{\sigma(A),\lambda h\},
\end{equation*}
\begin{equation*}
    t_{i+1}=\frac{(1-\lambda)^2t^2_i}{h}~~~~~~~ for~~~~~~ i\geq 1.
\end{equation*}
\noindent Then $\beta(A)\geq t_{n}$.
\end{corol}

\begin{corol}\label{cor3}
Let $A\in R^{n \times n}$ be an $H$-matrix with the positive diagonals. Let $\bar{A}$
be the comparison matrix of $A$, which is defined by
 \begin{equation*}
         \bar{A}_{ij} =
            \begin{cases}
             A_{ii} & \text{ if $i= j$},\\
               -|A_{ij}| & \text{if $i\neq j$ }.
         \end{cases}
        \end{equation*}\label{18}
        Then, for any vector $e>0$, the vector $e_1=\bar{A}^{-1}e>0$ and
         \begin{equation*}\label{eq18}
         \beta(A)\geq\frac{(\min_ie_i)(\min_i{e_1}_i)}{(\min_j{e_1}_j)^2}.
         \end{equation*}
         \end{corol}
         The upper and lower boundary of the term $\beta (A)$ are established by the above corollaries. Now we study the error bound related to diagonal $P$-matrix. 
         \begin{theorem}\label{corr1}
         Let $A\in R^{n \times n}$ be a diagonal $P$-matrix. Let $z\in R^n$ be the unique solution of LCP$(A,q)$ and $d\in R^n$ be an arbitrary vector. Then \\
          $\frac{1}{1+\|A\|_{\infty}}\|(d-(d-(Ad+q))_{+})\|_{\infty}\leq \|z-d\|_{\infty}\leq \frac{1+\|A\|_{\infty}} {\min_{i} A_{ii}}\|(d-(d-(Ad+q))_{+})\|_{\infty},$ where $A_{ii}$ is the $i$-th diagonal element of the matrix $A$.
        
         \end{theorem}
         \begin{proof}
         Let $A\in R^{n \times n}$ be a diagonal $P$-matrix. Then 
         \begin{center}
         $\max\limits_{i} A_{ii}=\max\limits_{\|x\|_{\infty}=1, i} A_{ii}{x_i}^2\geq \min\limits_{i} A_{ii}>0,$
         \end{center}
         where $A_{ii}$ is the $i$-th diagonal element of the matrix $A$.
         
         By definition \ref{def1},
         \begin{center}
          $\beta(A)=\min\limits_{\|z\|_{\infty}=1} \|z\|^{2}_A =\min\limits_{\|z\|_{\infty}=1} \max\limits_{ i }z_{i}(Az)_{i}\geq \min\limits_{i}A_{ii}$ 
          \end{center}
          and by theorem \ref{cor1},
          \begin{center}
              $\beta(A)\leq \sigma(A)$.
              \end{center}
              For diagonal $P$-matrix $A$, 
              \begin{center}
              $\sigma (A)=\min\limits_{i}A_{ii}.$
              \end{center}
              Hence  
              \begin{center}
          $\beta (A) \leq \min\limits_{i}A_{ii}$.
          \end{center}
           Both the inequalities imply that
           \begin{center}
           $\beta(A)=\min\limits_{i} A_{ii},$ 
           \end{center}
           where $A_{ii}$ is the $i$-th diagonal element of the matrix $A$. Let $z\in R^n$ be the unique solution of LCP$(A,q)$ and $d\in R^n$ be an arbitrary vector. From theorem \ref{lem2}, we obtain the following inequality\\
           \begin{center}
           $\frac{1}{1+\|A\|_{\infty}}\|(d-(d-(Ad+q))_{+})\|_{\infty}\leq \|z-d\|_{\infty}\leq \frac{1+\|A\|_{\infty}} {\beta(A)}\|(d-(d-(Ad+q))_{+})\|_{\infty}.$ 
           \end{center}
           Now using the value of  $\beta(A)$ in the inequality \ref{eqq8}, we obtain  the following inequality\\
           \begin{center}
         $\frac{1}{1+\|A\|_{\infty}}\|(d-(d-(Ad+q))_{+})\|_{\infty}\leq \|z-w\|_{\infty}\leq \frac{1+\|A\|_{\infty}} {\min\limits_{i} A_{ii}}\|(d-(d-(Ad+q))_{+})\|_{\infty}$,
         \end{center}
         where $A_{ii}$ is the $i$-th diagonal element of the matrix $A$.
         \end{proof}
         \begin{theorem}
         Let $A\in R^{n \times n}$ be a diagonal $P$-matrix. Let $z\in R^n$ be the unique solution of LCP$(A,q)$ and $d\in R^n$ be an arbitrary vector. Then the relative error satisfies the following inequality\\
          \begin{equation*}
    \frac{{\min\limits_{i} A_{ii}}}{1+\|A\|_{\infty}}\frac{\|(d-(d-(Ad+q))_{+})\|_{\infty}}{\|( - q)_+\|_{\infty}}
\end{equation*}

\begin{equation}\label{eq9}
\leq \frac{\|z-d\|_{\infty}}{\|z\|_{\infty}}\leq \frac{1+\|A\|_{\infty}}{(\frac{1}{\max\limits_{i} A_{ii}}){\min\limits_{i} A_{ii}}}\frac{\|(d-(d-(Ad+q))_{+})\|_{\infty}}{\|( - q)_+\|_{\infty}},
\end{equation} where $A_{ii}$ is the $i$-th diagonal element of the matrix $A$.
\end{theorem}
\begin{proof}
 Let $A\in R^{n \times n}$ be a diagonal $P$-matrix. By theorem \ref{corr1}, it is clear that $\beta(A)=\min\limits_{i} A_{ii},$ where $A_{ii}$ is the $i$-th diagonal element of the matrix $A$. Since $A$ is a diagonal matrix, $\beta(A^{-1})=\min\limits_{i} ({A^{-1}})_{ii}=\frac{1}{\max\limits_{i} A_{ii}}.$ Now from theorem \ref{lem3}, we obtain \\
  \begin{equation*}
    \frac{{\min\limits_{i} A_{ii}}}{1+\|A\|_{\infty}}\frac{\|(d-(d-(Ad+q))_{+})\|_{\infty}}{\|( - q)_+\|_{\infty}}
\end{equation*}

\begin{equation}\label{eq9}
\leq \frac{\|z-d\|_{\infty}}{\|z\|_{\infty}}\leq \frac{1+\|A\|_{\infty}}{(\frac{1}{\max\limits_{i} A_{ii}}){\min\limits_{i} A_{ii}}}\frac{\|(d-(d-(Ad+q))_{+})\|_{\infty}}{\|( - q)_+\|_{\infty}},
\end{equation} where $z\in R^n$ be the unique solution of LCP$(A,q)$, $d\in R^n$ be an arbitrary vector and $A_{ii}$ is the $i$-th diagonal element of the matrix $A$.
\end{proof}
\section{Numerical Example}
Consider the matrix $A=\left[\begin{array}{ccc}
4 & 1 & 2\\
3 & 5 & -1\\
-1 & -2 & 7\\
\end{array}\right]$, which is a $P$-matrix.\\ The principal submatrices of $A$ are $A_{11}=4, A_{22}=5, A_{33}=7,$ $ A_{\alpha \alpha}=\left[\begin{array}{cc}
4 & 1\\
3 & 5 \\
\end{array}\right]$, where $\alpha =\{1,2\},$
$ A_{\beta \beta}=\left[\begin{array}{cc}
5 & -1\\
-2 & 7 \\
\end{array}\right]$, where $\beta =\{2,3\},$
 $ A_{\delta \delta}=\left[\begin{array}{cc}
4 & 2\\
-1 & 7 \\
\end{array}\right]$, where $\delta =\{1,3\}.$
\vsp
\NI Now $\gamma(A_{11})=4, \gamma(A_{22})=5, \gamma(A_{33})=7, \gamma(A_{\alpha \alpha})=\frac{9-\sqrt{13}}{2}, \gamma(A_{\beta \beta})=\frac{12-\sqrt{12}}{2}, \gamma(A_{\delta \delta})=4.5, $ where $\gamma $ is defined by theorem \ref{cor1}. \\
\vsp
\NI Hence $\sigma(A)= \min\{4, 5, 7, 4.5, {\frac{9-\sqrt{13}}{2}}, \frac{12-\sqrt{12}}{2}\}= \frac{9-\sqrt{13}}{2}=2.69722436.$\\ Now $m=\max_{i\neq j}|A_{ij}|=3, h=\frac{m^2}{\sigma(A)}=3.33676357.$\\
Let $\lambda=0.5$,
then \\$t_1=\min\{\sigma(A),\lambda h\}$=$\min\{2.69722436,1.66838179\}=1.66838179$.\\$t_2=\frac{(1-\lambda)^2t^2_1}{h}=\frac{0.25*2.7834978}{3.33676357}=0.208547725.$ \\  $t_3=\frac{(1-\lambda)^2t^2_2}{h}=\frac{0.25*0.0434921536}{3.33676357}=0.00325855823.$\\
Therefore $\beta(A)\geq 0.00325855823.$
Again, $\beta(A)\leq {\|z\|_A}^2\leq \sigma(A)=2.69722436.$\\
Here $\|A\|_{\infty}=7$.
Let $z\in R^n$ be the unique solution of LCP$(A,q)$ and $d\in R^n$ be an arbitrary vector. From \ref{eqq8}, we obtain
\begin{center}
$\frac{1}{1+\|A\|_{\infty}}\|(d-(d-(Ad+q))_{+})\|_{\infty}\leq \|z-d\|_{\infty}\leq \frac{1+\|A\|_{\infty}} {\beta(A)}\|(d-(d-(Ad+q))_{+})\|_{\infty}$.
\end{center}
This implies that
\begin{center}
$\frac{1}{8}\|(d-(d-(Ad+q))_{+})\|_{\infty}\leq \|z-d\|_{\infty}\leq \frac{8} {2.69722436}\|(d-(d-(Ad+q))_{+})\|_{\infty}$.
\end{center}
Therefore the error satisfies the inequality,
\begin{center}
$0.125\|(d-(d-(Ad+q))_{+})\|_{\infty}\leq \|z-d\|_{\infty}\leq 2.96601207\|(d-(d-(Ad+q))_{+})\|_{\infty}$. 
\end{center}

\section{Conclusion}
In this study we introduce error bound for LCP$(A,q)$ with $P$-matrix using plus function. We introduce a new residual approach to bound the error as well as the relative error. We also study the error bound for diagonal $P$-matrix. A numerical example is illustrated to demonstrate the upper and lower bound of the error.

\section{Acknowledgment}
 The author Bharat Kumar is thankful to the University Grant Commission (UGC),  Government of India under the JRF programme no.  1068/(CSIR-UGC NET DEC. 2017 ). The author A. Dutta is thankful to the Department of Science and Technology, Govt. of India, INSPIRE Fellowship Scheme for financial support.

\bibliographystyle{plain}
\bibliography{name}
\end{document}